\documentclass{amsart}
\usepackage{latexsym}
\usepackage{amsmath}
\usepackage{amssymb}
\usepackage[all]{xy}

\newtheorem{thm}{Theorem}[section]
\newtheorem{prop}[thm]{Proposition}

\newtheorem{lem}[thm]{Lemma}
\theoremstyle{definition}
\newtheorem{defn}[thm]{Definition}

\theoremstyle{remark}
\newtheorem{rem}[thm]{Remark}
\newtheorem{ex}[thm]{Example}

\newcommand{\h}[1]{%
 H^*({#1}; {\mathbb K})}
\newcommand{\K}{{\mathbb K}}
\newcommand{\Q}{{\mathbb Q}}
\newcommand{\T}{{\mathcal T}}
\newcommand{\e}{\varepsilon}
\newcommand{\D}{\text{D}}
\newcommand{\DD}{\text{\em D}}
\newcommand{\mapright}[1]{%
 \smash{\mathop{%
  \hbox to 1cm{\rightarrowfill}}\limits_{#1}}}
\newcommand{\maprightd}[2]{%
 \smash{\mathop{%
  \hbox to 1.2cm{\rightarrowfill}}\limits^{#1}\limits_{#2}}}
\newcommand{\mapleft}[1]{%
 \smash{\mathop{%
  \hbox to 1cm{\leftarrowfill}}\limits_{#1}}}
\newcommand{\mapleftu}[1]{%
 \smash{\mathop{%
  \hbox to 0.8cm{\leftarrowfill}}\limits^{#1}}}
\newcommand{\maprightu}[1]{%
 \smash{\mathop{%
  \hbox to 1cm{\rightarrowfill}}\limits^{#1}}}
\newcommand{\maprightud}[2]{%
 \smash{\mathop{%
  \hbox to 1cm{\rightarrowfill}}\limits^{#1}_{#2}}}
\newcommand{\mapleftud}[2]{%
 \smash{\mathop{%
  \hbox to 1cm{\leftarrowfill}}\limits^{#1}_{#2}}}

\newcounter{eqn}[section]

\def\theeqn{\textnormal{(\thesection.\arabic{eqn})}}

\def\eqnlabel#1{%
  \refstepcounter{eqn}%
  \label{#1}%
  \leqno{\theeqn}}

\begin{document}

\title[On the levels of maps]{On the levels of maps and topological realization
of objects in a triangulated category}

\footnote[0]{


Department of Mathematical Sciences, 
Faculty of Science,  
Shinshu University,   
Matsumoto, Nagano 390-8621, Japan   
e-mail:{\tt kuri@math.shinshu-u.ac.jp}
}

\author{Katsuhiko KURIBAYASHI}
\date{}
   
\maketitle

\begin{abstract}
The level of a module over a differential graded algebra measures 
the number of steps required to build the module 
in an appropriate triangulated category.   Based on this notion, 
we introduce a new homotopy invariant of spaces over a fixed space, called the level of a map. 
Moreover we provide a method to compute the invariant for spaces over a $\K$-formal space.  
This enables us to determine  the level of the total space of a bundle over 
the $4$-dimensional sphere with the aid of Auslander-Reiten theory for spaces due to J{\o}rgensen.   
We also discuss the problem of realizing 
an indecomposable object in the derived category of the sphere 
by the singular cochain complex of a space. 
The Hopf invariant provides a criterion for the realization. 
\end{abstract}

\section{Introduction}
Categorical representation theory yields suitable tools for studying 
certain problems 
in finite group theory, algebraic geometry and algebraic topology. 
For example, the Auslander-Reiten quiver of a triangulated category is 
an interesting combinatorial invariant; 
see \cite{H}, \cite{H2}, \cite{J}, \cite{J2} and \cite{S}. 
The singular (co)chain complex functor is a necessary ingredient in 
developing algebraic model theory for topological spaces; see \cite{A}, 
\cite{B-L}, \cite{FHT2}, \cite{H-L} and \cite{M}.  
We will here advertise the idea that this functor, combined with tools
from categorical representation theory of the kind just mentioned, is likely
to provide new insights into the relationship between algebra and
topology. To this end, we introduce and study a homotopy 
invariant that we call the {\it level}  of a map.

The notion of levels of objects in a triangulated category 
was originally introduced 
by Avramov, Buchweitz, Iyengar and Miller in \cite{ABIM}. 
Roughly speaking, the level of an object $M$ in a triangulated category
 $\T$ counts the number of steps required to build $M$ out of a fixed 
 object via triangles in $\T$.  

Let $X$ be a space and ${\mathcal TOP}_X$ the category of spaces over $X$. 
The singular cochain  complex functor $C^*( \ ; \K)$ 
with coefficients in a field 
$\K$ gives rise to a contravariant functor 
from ${\mathcal TOP}_X$ to the derived
category $\D(C^*(X; \K))$ of DG (that is, differential graded) modules 
over the DG algebra $C^*(X; \K)$. Observe that  $\D(C^*(X; \K))$  is a 
triangulated category with shift functor $\Sigma$; $(\Sigma M)^n = M^{n+1}$.   
We then define the level of a space $Y$ over $X$ to be the level of 
the DG $C^*(X; \K)$-module $C^*(Y; \K)$;   
see Section 2 for the exact definition.  

In the rest of this section, we survey our main results.  

After showing that the level of a space is a weak homotopy invariant on ${\mathcal TOP}_X$, 
we give a reduction theorem (Theorem \ref{thm:main}) for
computing the level of a pullback of $\K$-formal spaces. 
An explicit calculation using this theorem tells us that a `nice' space 
such as the total space $E$ of a bundle over the sphere $S^d$
is of low level; 
see Propositions \ref{prop:bundle} and \ref{prop:variation}. 
This means that the object $C^*(E; \K)$ in $\D(C^*(S^d; \K))$ is built out of indecomposable objects of 
low level in the full subcategory
 of compact objects $\D^c(C^*(S^d; \K))$.  
These indecomposable objects, which we call {\it molecules} of $C^*(E; \K)$, 
are visualized with black vertices in the Auslander-Reiten quiver of 
$\D^c(C^*(S^d; \K))$ as drawn below. 
\vspace{-0.2cm}
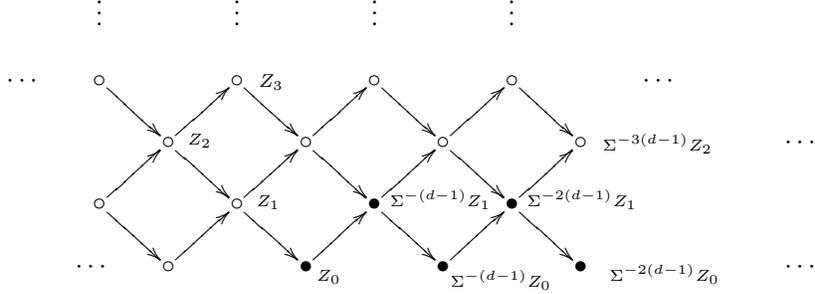
$$
\objectmargin={0.5pt}
\xymatrix@C20pt@R18pt{
&  \vdots & {}\ar@{}[rd]^(1.2){Z_3} & \vdots & & \vdots & & \vdots \\
\cdots & \circ \ar[rd]^(1.2){\mbox{\tiny$Z_2$}} &   & \circ \ar[rd]^(1.3){} & & \circ
                       \ar[rd]
        & & \circ  \ar[rd]^(1.3){{\mbox{\tiny{$\Sigma^{-3(d-1)}Z_2$}}}} & & \cdots \\
&   &\circ\ar[rd]^(1.2){\mbox{\tiny$Z_1$}}\ar[ru]  &  & \circ \ar[rd]^(1.2)
{{\mbox{\tiny{$\Sigma^{-(d-1)}Z_1$}}}} \ar[ru] & 
               & \circ\ar[rd]^(1.2){{\mbox{\tiny{$\Sigma^{-2(d-1)}Z_1$}}}} \ar[ru] & & \circ&
   & & \cdots \\
& \circ\ar[rd]\ar[ru]    &  & \circ \ar[rd]\ar[ru] &  &
                       \bullet\ar[rd] \ar[ru]&  & \bullet
 \ar[rd]^(1.4){{\mbox{\tiny{$\Sigma^{-2(d-1)}Z_0$}}}}\ar[ru] & & \\ 
& \!\!\!\! \cdots  &\circ \ar[ru] &   & \bullet \ar[ur]_(0.08){\mbox{\tiny$Z_0$}} & 
& \bullet\ar[ru]_(0.08){{\mbox{\tiny{$\Sigma^{-(d-1)}Z_0$}}}}
 & & \bullet  & & & \cdots 
}
$$
Here only the component of the quiver containing $Z_0=C^*(S^d; \K)$ is illustrated.  
Thus one has a new algebraic aspect of a topological object.  
For more details of the Auslander-Reiten quiver of a space, we refer the reader to Theorem \ref{thm:J}, 
which is a remarkable result due to J{\o}rgensen.

The level of a map $Y \to B$ provides a lower bound on the number of spherical 
fibrations required to construct $Y$ from $B$; 
see Proposition \ref{prop:pile} and Theorem \ref{thm:ex-level}. 
A topological description of the level is here given.  
Moreover, Theorem \ref{thm:ex-level} and Proposition \ref{prop:Z} imply that there exists at least 
one molecule in each row of the the Auslander-Reiten quiver of $\D^c(C^*(S^d; \Q))$  
which is a summand of $C^*(X; \Q)$ for some space $X$ over $S^d$.   

Intriguing properties of the notion level are investigated in followups to this article 
\cite{K3} \cite{K4}. In particular, we show in \cite{K3} that the dual, {\it chain-type level} 
of a map $f : X \to Y$ provides an upper bound on the Lusternik-Schnirelmann category of $X$, at least over 
$\Q$. In \cite{K4} we explain that cochain-type  and chain-type levels are related by a sort of Eckmann-Hilton duality. 

We deal with the problem of realizing a vertex (molecule) in 
an Auslander-Reiten quiver by 
the singular cochain complex of a space.
It turns out that almost all molecules which appear  in the quiver 
over the sphere are not realized by finite CW complexes. 
In fact Theorem \ref{thm:realization} states that, in the
Auslander-Reiten quiver mentioned above, only the arrow 
$$
\xymatrix@C20pt@R5pt{
\mbox{\small{$Z_0$}}\ \  \bullet \ar[r] & \bullet \ \
\mbox{\small{$\Sigma^{-(d-1)}Z_1$}}}
$$
is realizable.  
Proposition \ref{prop:spheres} asserts that a map 
$\phi : S^d \to S^{2d-1}$ realizes 
the arrow if and only if the Hopf invariant of $\phi$ is non-trivial. 
This gives a new
topological perspective on the Auslander-Reiten quiver. 

Statements of all our results can be found in Section 2, 
while the proofs are in sections 3 through 7. 

\section{Results}

We fix some terminology. Throughout this article differential graded objects are written in 
the cohomological notation; that is, the differential increases degree by $1$.  
We say that a graded vector space $M$ is {\it locally finite} if $M^i$ is
of finite dimension for any $i$.  
Moreover $M$ is said to be {\it non-negative} if $M^i=0$ for $i<0$.  
A DG algebra $A$ over a field $\K$ is {\it simply-connected} if it is non-negative and 
satisfies the condition that $H^0(A)=\K$ and $H^1(A)=0$. 
We refer to  a morphism between DG $A$-modules as a {\it quasi-isomorphism} 
if it induces an isomorphism on the homology. 
Note that unspecified DG $A$-modules are right DG $A$-module.  
Unless otherwise explicitly stated, it is assumed that a space has the
homotopy type of a CW complex whose cohomology with coefficients in the
underlying field is locally finite. 
Observe that the cochain algebra $C^*(X; \K)$ of a simply-connected space $X$ 
is simply-connected. 

The goal of this section is to state our results in more detail.  

Let $\T$ be a  triangulated category. 
To introduce the notion of the level, 
we first recall from \cite{ABIM} the
definition of the thickening of  $\T$. 
For a given object $C$ in $\T$, 
we define the $0$th thickening by $\text{{\tt thick}}^0_{\T}(C)=\{0\}$
and $\text{{\tt thick}}^1_{\T}(C)$ by the smallest strict full subcategory
which contains $C$ and is closed under taking finite coproducts,
retracts and all shifts. Moreover for $n > 1$ define inductively 
the $n$th thickening $\text{{\tt thick}}^n_{\T}(C)$ 
to be the smallest strict full subcategory of 
$\T$ which is closed under retracts and contains objects $M$
admitting an exact triangle 
$$
M_1 \to M \to M_2 \to \Sigma M_1
$$
in $\T$ for which $M_1$ and $M_2$ are in 
$\text{{\tt thick}}^{n-1}_{\T}(C)$ and $\text{{\tt thick}}^1_{\T}(C)$,
respectively. 

By definition, a full subcategory ${\mathcal C}$ of $\T$ is {\it thick} 
if it is additive, closed under retracts, 
and every exact triangle in $\T$ with two vertices in  ${\mathcal C}$  has its third vertex in ${\mathcal C}$.   
As mentioned in \cite[2.2.4]{ABIM}, the thickenings provide a filtration
of the smallest thick subcategory $\text{{\tt thick}}_{\T}(C)$ of $\T$ containing the
object $C$: 
$$
\{0\} = \text{{\tt thick}}^0_{\T}(C) \subset \cdots \subset 
\text{{\tt thick}}^n_{\T}(C) \subset \cdots \subset 
\cup_{n\geq 0}\text{{\tt thick}}^n_{\T}(C) = \text{{\tt thick}}_{\T}(C). 
$$

For an object $M$ in $\T$, 
we define a numerical invariant $\text{level}_{\T}^C(M)$, which is
called the {\it $C$-level of} $M$,  by 
$$
\text{level}_{\T}^C(M):= \inf \{n \in {\mathbb N}\cup \{0\} \ |
 M \in  \text{{\tt thick}}^{n}_{\T}(C) \}. 
$$
It is worth noting that $\text{level}_{\T}^C(M)$ is finite if and only if $M$  is 
{\it finitely built from} $C$ in the sense of Dwyer, Greenlees and
Iyenger \cite[3.15]{D-G-I}; see also \cite{D-G}.

Let $A$ be a DG algebra over a field $\K$. 
Let $\D(A)$ be the derived category of DG $A$-modules, namely the localization of the homotopy category 
$\text{H}(A)$ of DG $A$-modules with respect to quasi-isomorphisms; see \cite{Keller} and \cite[PART III]{KM}.   
Observe that $\D(A)$ is a triangulated category with the shift functor 
$\Sigma$ defined by $(\Sigma M)^n=M^{n+1}$ and that 
a triangle in $\D(A)$ comes from a cofibre sequence of the form 
$
M \stackrel{f}{\to} N \to C_f \to \Sigma M
$ 
in the homotopy category $\text{H}(A)$. 
Here $C_f$ denotes the mapping cone of $f$. 
In what follows, 
for any object $M$ in $\D(A)$, we may write 
 $\text{level}_{\D(A)}(M)$ for the $A$-level $\text{level}_{\D(A)}^A(M)$ of $M$.

Let $X$ be a simply-connected space and $\mathcal{TOP}_X$ the category
of connected spaces over $X$; that is, 
objects are maps to the space $X$ and 
morphisms from $\alpha : Y \to X$ to $\beta : Z \to X$ are maps $f : Y
\to Z$ such that $\beta f = \alpha$. 
For an object $\alpha: Y \to X$ in $\mathcal{TOP}_X$, 
the singular cochain complex $C^*(Y; \K)$ is considered  
a DG module over the DG algebra $C^*(X; \K)$ via the morphism of DG algebras induced by 
$\alpha$. 
We may write $C^*(Y; \K)^{\alpha}$ for this DG-module.   
Thus we have a contravariant functor  
$$
C^*( \ ; \K) : \mathcal{TOP}_X \to \D(C^*(X; \K)). 
$$
\begin{defn}
Let $\alpha : Y \to X$ be an object in $\mathcal{TOP}_X$.  
The {\it level} of the map $\alpha$, denoted 
$\text{level}_{X}(Y)_\K$, is the $C^*(X; \K)$-level of $C^*(Y; \K)^{\alpha}$ in the triangulated category 
$\D(C^*(X; \K))$, namely 
$\text{level}_{\D(C^*(X; \K))}^{C^*(X; \K)}(C^*(Y; \K)^{\alpha})$.
\end{defn}
   
When there is no danger of confusion, we will write $\text{level}_{X}(Y)$ in place of $\text{level}_{X}(Y)_\K$.  
Note that, in \cite{K3}, we call the level of a map $\alpha :  Y \to X$ 
the {\it cochain type level} of the {\it space} $Y$ and write 
$\text{level}_{\text{D}(C^*(X; \K))}(Y)$ for $\text{level}_{X}(Y)_{\K}$.   

A straightforward argument shows that the level is a weak homotopy invariant on 
$\mathcal{TOP}_X$. 

\begin{prop} \label{prop:inv}
Let $\alpha : Y \to X$ and $\beta : Z \to X$ be objects in $\mathcal{TOP}_X$. 
If there exists a weak homotopy equivalence $f  : Y \to Z$ such that $\alpha \simeq \beta\circ f$, 
then $$
\text{\em level}_{X}(Z)=\text{\em level}_{X}(Y). 
$$
\end{prop}

\begin{proof} 
Let $H : Y \times I \to X$ be a homotopy from $\alpha$ to $\beta\circ f$ 
and $\e_i : Y \to Y\times I$ the inclusion defined by $\e(y) = (y, i)$ for $i = 0, 1$.  
We consider $C^*(Y\times I; \K)$ a DG $C^*(X;\K)$-module via the induced map 
$H^* : C^*(X;\K) \to C^*(Y\times I; \K)$.  Moreover $C^*(Y;\K)$ is endowed with a DG $C^*(X;\K)$-module 
structure via the map $(H\circ \e_i)^* : C^*(X;\K) \to C^*(Y;\K)$ for each $i = 0, 1$. 
Then there exists a sequence of quasi-isomorphisms of DG $C^*(X;\K)$-modules
$$
\xymatrix@C15pt@R15pt{
C^*(Z;\K)^{\beta} \ar[r]^(0.45){f^*}_(0.45)\simeq & C^*(Y;\K)^{H\circ \e_1} & C^*(Y\times I;\K)^{H} \ar[l]_{\e_1^*}^{\simeq} 
\ar[r]^(0.35){\e_0^*}_(0.35){\simeq} & C^*(Y;\K)^{H\circ \e_0} = C^*(Y;\K)^\alpha. 
}
$$
Thus we have the result. 
\end{proof}


It is natural to ask what aspect of topological spaces is captured by the notion of level. 
To begin to answer this question, it is helpful to compute the level of various interesting maps. 
As an aid to computation we provide a reduction theorem for levels of certain maps of 
$\K$-formal spaces.

Let $m_X : TV_X \stackrel{\simeq}{\to} C^*(X; {\mathbb K})$ 
be a minimal TV-model for a simply-connected space 
in the sense of Halperin and Lemaire \cite{H-L}; 
that is, $TV_X$ is a DG algebra whose underlying $\K$-algebra 
is the tensor algebra generated by a graded vector space $V_X$ and, for
any element $v \in V_X$, the image of $v$ by the differential is
decomposable; see also Appendix.

Recall that a space $X$ is {\it $\K$-formal} 
if it is simply-connected and there exists a sequence of quasi-isomorphisms 
of DG algebras 
$$
\xymatrix@C25pt@R15pt{
\h{X} & TV_X  \ar[l]_(0.4){\phi_X}^(0.4){\simeq} 
\ar[r]^(0.4){m_X}_(0.4){\simeq}  & C^*(X; {\mathbb K}),
} 
$$
where $m_X : TV_X \to C^*(X; {\mathbb K})$ denotes a minimal 
$TV$-model for $X$. 
Observe that spheres if $d > 1$, then the sphere $S^d$ is  $\K$-formal,  
for any field $\K$ \cite{E}\cite{B-T}. Moreover 
a simply-connected space whose cohomology with coefficients in $\K$ 
is a polynomial algebra generated by elements of even degree is
$\K$-formal \cite[Section 7]{Mu}.  

\begin{defn}
\label{defn:formalizable}
Let $q : E \to B$ and $f : X \to B$ be maps between 
$\K$-formal spaces. The pair $(q, f)$ is {\it relatively 
$\K$-formalizable} if there exists a commutative diagram up to homotopy of DG algebras 
$$
\xymatrix@C25pt@R20pt{
\h{E}  & TV_E \ar[l]_(0.4){\phi_E}^(0.4){\simeq} \ar[r]^(0.4){m_E}_(0.4){\simeq} 
    & C^*(E; {\mathbb K})  \\
\h{B} \ar[u]^{H^*(q)} \ar[d]_{H^*(f)}
  & TV_B  \ar[l]_(0.4){\phi_B}^(0.4){\simeq} \ar[r]^(0.4){m_B}_(0.4){\simeq} 
  \ar[u]_{\widetilde{q}} \ar[d]^{\widetilde{f}}
& C^*(B; {\mathbb K}) \ar[u]_{q^*} \ar[d]^{f^*} \\ 
\h{X} & TV_X  \ar[l]_(0.4){\phi_X}^(0.4){\simeq} \ar[r]^(0.4){m_X}_(0.4){\simeq} 
 & C^*(X; {\mathbb K}) ,
}
$$
in which horizontal arrows are quasi-isomorphisms.  
\end{defn}

In general, for given quasi-isomorphisms $\phi_E$, $m_E$, $\phi_B$ and $m_B$ 
as in Definition \ref{defn:formalizable}, there exist DG algebra  maps  
$\widetilde{q}_1$ and $\widetilde{q}_2$ which make the right upper square and
left one  homotopy commutative, respectively. 
However, in general, one cannot choose a map $\widetilde{q}$
which makes upper two squares homotopy commutative simultaneously even
if the maps $\phi_E$, $m_E$, $\phi_B$ and $m_B$  are replaced by other
quasi-isomorphisms;  see Remark \ref{rem:EMSS-K-formal}. 

The following proposition, which is deduced 
from the proof of \cite[Theorem 1.1]{K}, gives examples of 
relatively $\K$-formalizable pairs of maps. 

\begin{prop}
\label{prop:formalizable} 
A pair of maps $(q, f)$ with a common target is relatively 
$\K$-formalizable if each of the maps satisfies 
either of the two conditions below on a map $\pi : S \to T$. 

\medskip
\noindent
{\em(i)} $\h{S}$ and $\h{T}$ are polynomial algebras with at most countably
many generators in which the operation 
$Sq_1$ vanishes when the characteristic of the field $\K$ is 2. 
Here $Sq_1x = Sq^{n-1}x$ for $x$ of degree $n$; see \cite[4.9]{Mu}. 

\noindent
{\em (ii)} $\widetilde{H}^i(S; {\mathbb K})=0$ 
for any $i$ with 
$\dim\widetilde{H}^{i-1}(\Omega T; {\mathbb K}) 
 - \dim (QH^*(T;  {\mathbb K}))^i \neq 0$.   
\end{prop}

Let $q : E \to B$ be a fibration over a space $B$ and 
$f : X \to B$ a map. 
Let ${\mathcal F}$ denote the pullback diagram
$$
\xymatrix@C20pt@R25pt{
E \times _B X \ar[r] \ar[d] & E \ar[d]^{q} \\
X  \ar[r]_{f}  & B   .
}
$$
Our main theorem on the computation of the level of a space is stated as follows.   

\begin{thm}
\label{thm:main} Suppose that the spaces $X$, $B$ and $E$ 
in the diagram ${\mathcal F}$ are $\K$-formal  
and the pair $(q, f)$ 
is relatively $\K$-formalizable. Then 
$$
\text{\em level}_{X}(E\times_B X) 
= \text{\em level}_{\text{\em D}(H^*(X; \K))}
(H^*(E;\K)\otimes_{H^*(B;\K)}^\text{\em L}H^*(X; \K)).  
$$
\end{thm}

As Example \ref{ex:S7} illustrates, the condition that $X$, $B$ and $E$ in ${\mathcal F}$ 
are  $\K$-formal is not sufficient. 
We refer the reader to Section 3 for the definition of the left derived functor $-\otimes^L -$.

By virtue of Theorem \ref{thm:main} and 
Proposition \ref{prop:formalizable}, we have  

\begin{prop}
\label{prop:bundle} Let $G$ be a simply-connected Lie group and 
$G \to E_f \to S^4$ a $G$-bundle with the classifying map  
$f : S^4 \to BG$. Suppose that $H^*(BG; \K)$ is a polynomial algebra
 on generators of even degree. Then 
$$
\text{\em level}_{S^4}(E_f)=
\left\{
\begin{array}{l}
2  \ \ \text{if} \ H^4(f; \K)\neq 0,  \\
1  \ \ \text{otherwise}. 
\end{array}
\right.
$$ 
\end{prop}

\begin{prop}
\label{prop:variation} Let $G$ be a simply-connected 
Lie group and $H$ a maximal rank  subgroup. 
Let $G/H \to E_g \to S^4$ be the pullback of the fibration 
$G/H \to BH \stackrel{\pi}{\to} BG$ by a map $g : S^4 \to BG$. 
  Suppose that $H^*(BG; \K)$ and  $H^*(BH; \K)$ are polynomial algebras
 on generators with even degree. Then 
$$
\text{\em level}_{S^4}(E_g)=1.
$$ 
\end{prop}

As an introduction of the meaning of the level of a maps $f$, we show that it provides an lower bound 
on the number of stages in a factorization 
$$
Y=Y_c \stackrel{\pi_c}{\longrightarrow} Y_{c-1} \stackrel{\pi_{c-1}}{\longrightarrow} \cdots  \stackrel{\pi_2}{\longrightarrow} Y_{1} 
\stackrel{\pi_1}{\longrightarrow} Y_{0} \stackrel{\pi_0}{\longrightarrow} B  
$$
of $f$, where each $\pi_i$ is a fibration with an odd sphere as fibre.

\begin{prop}
\label{prop:pile}
Suppose that there exists a sequence of fibrations 
\begin{eqnarray*}
S^{2m_1+1} \to Y_1 \stackrel{\pi_1}{\longrightarrow} B
 \times (\displaystyle{\times_{i=1}^sS^{2n_i+1}}), & 
S^{2m_2+1}  \to Y_2 \stackrel{\pi_2}{\longrightarrow} Y_1, \  ....., \\  
S^{2m_c+1} \to Y_c \stackrel{\pi_c}{\longrightarrow} Y_{c-1} &  
\end{eqnarray*}
in which $B$ is simply-connected and $n_i, m_j \geq 1$ 
for any $i$ and $j$. We regard $Y_c$ as a space over $B$ via 
the composite $\pi_0\circ \pi_1 \cdots \circ \pi_c$, where  
$\pi_0 : B \times (\displaystyle{\times_{i=1}^lS^{2n_i+1}}) \to B$ 
is the projection onto the first factor.     
Then 
$$
\text{\em level}_{B}(Y_c)_{\Q} \leq c+1.
$$
\end{prop}

By using Proposition \ref{prop:pile} and the homological information of
each vertex of the Auslander-Reiten quiver of 
$\text{D}(C^*(S^d; \K))$ described in Theorem \ref{thm:J}, 
we can construct an object in $\mathcal{TOP}_{S^d}$ of arbitrary level,  
provided that $\K=\Q$.

\begin{thm}
\label{thm:ex-level} 
For any integers $l\geq 1$ and $d > 1$, there exists an object $P_l \to S^d$
 in $\mathcal{TOP}_{S^d}$ such that 
$$
\text{\em level}_{S^d}(P_l)_{\Q}= l. 
$$ 
\end{thm}

The map $P_l \to S^d$ in the statement above is constructed iteratively by spherical fibrations, 
as in Proposition \ref{prop:pile}.  

\medskip
Proposition \ref{prop:pile} also clarifies 
a link between the level of a rational space 
$X$ and the codimension of $X$ due to
Greenlees, Hess and Shamir \cite{G-H-S}. 

\begin{defn}\cite[7.4(i)]{G-H-S} A space $X$ is spherically complete 
intersection (sci) if it is simply-connected and 
there exists a sequence of spherical fibrations 
\begin{eqnarray*}
S^{m_1} \longrightarrow X_1 {\longrightarrow} KV,  & S^{m_2}  \longrightarrow X_2 {\longrightarrow} Y_1,  \ ..., \ 
S^{m_c} \longrightarrow X_c {\longrightarrow} X_{c-1}  & 
\end{eqnarray*}
in which $X_c = X$ and $KV$ is a regular space, namely  
the Eilenberg-MacLane space on a finite dimensional graded vector space 
$V$ with $V^{odd}=0$. 
The least such integer $c$ is called the {\it codimension of} $X$, denoted 
$\text{codim}(X)$. 
\end{defn}
 
The result \cite[Lemma 8.1]{G-H-S} asserts that 
the spheres which appear in the definition of a sci space 
may be taken to be of odd dimension by replacing the regular space $KV$ 
by another regular space. Thus if $X$ is sci, by composing the projections in 
the fibrations, we have a new fibration  
$F \to X \stackrel{\pi}{\to} KV$ such that   
$$
\text{codim}(X)= \dim\pi_*(F)\otimes \Q= \dim \pi_{odd}(X)\otimes \Q.
$$
We call this fibration a standard fibration of $X$. 
Proposition \ref{prop:pile} yields immediately the following result.

\begin{thm}\label{thm:sci}
Let $X$ be sci with a standard fibration of the form $F \to X \to KV$.   
Then one has 
$$
\text{\em level}_{KV}(X)_{\Q} \leq \text{\em codim}(X)+1.
$$
\end{thm}

We next focus on the problem of realizing objects in
the triangulated category $\D(C^*(S^d; \K))$ 
as the singular cochain complexes of spaces. 
To this end, we describe J{\o}rgensen's result in \cite{J}  briefly.

Let $\T$ be a triangulated category. An object in $\T$ is said to be 
indecomposable if it is not a coproduct
of nontrivial objects. 
Recall that a triangle 
$$
L \stackrel{u}{\to} M  \stackrel{v}{\to} 
N  \stackrel{w}{\to} \Sigma L
$$ 
in $\T$ is an {\it  Auslander-Reiten triangle}  \cite{H}\cite{H2}   
if the following conditions are satisfied: \\

(i) $L$ and $N$ are indecomposable. 

(ii) $w\neq 0$.  

(iii) Each morphism $N' \to N$ which is not a retraction factors through
$v$. \\

 
We say that a morphism $f : M \to N$ in $\T$ is {\it irreducible} if it is
neither a section nor a retraction, but satisfies that in any
factorization $f=rs$, either $s$ is a section or $r$ is a retraction. 

The category $\T$ is said to have Auslander-Reiten triangles 
if, for each object $N$ whose endomorphism ring is local, there exists an 
Auslander-Reiten triangle with $N$ as the third term from the left. 
Recall also that an object $K$ in $\T$ is {\it compact} 
if the functor $\text{Hom}_{\T}(K, \ )$ preserves coproducts; see
\cite[Chapter 4]{N}.

\begin{defn}
The Auslander-Reiten quiver of $\T$ has as vertices the isomorphism
 classes $[M]$ of indecomposable objects. It has one arrow from 
$[M]$ to $[N]$ when there is an irreducible morphism $M \to N$
and no arrow from $[M]$ to $[N]$ otherwise. 
\end{defn}

Let $A$ be a locally finite, simply-connected DG algebra over a field $\K$. 
Assume further that $\dim H^*(A) < \infty$. 
We denote by $\D^c(A)$ the full subcategory of 
the derived category $\D(A)$ consisting of the compact objects.  
For a DG $A$-module $M$, let $DM$ be the dual 
$\text{Hom}_\K(M, \K)$ to $M$.

Put $d:= \sup \{i \ | \ H^iA \neq 0 \}$.  
One of the main results in \cite{J} asserts that both $\D^c(A)$ and 
$\D^c(A^{op})$ have Auslander-Reiten triangles if and only if there are
isomorphisms of graded $H^*A$-modules $_{H^*A}(DH^*A)\cong {}_{H^*A}(\Sigma^dH^*A)$
and $(DH^*A)_{H^*A}\cong \ (\Sigma^dH^*A)_{H^*A}$; that is,  $H^*(A)$ is a
Poincar\'e duality algebra.  In other words, 
$A$ is Gorenstein in the sense of F\'elix, Halperin and Thomas 
\cite{FHT_G}. In this case, the form 
of the Auslander-Reiten quiver of $\D^c(A)$ 
was determined in \cite{J} and \cite{J2}. 

The key lemma \cite[Lemma 8.4]{J} for 
proving results in \cite[Section 8]{J} is obtained 
by using the rational formality of the spheres. 
Since the spheres are also $\K$-formal for
any field $\K$, the assumption concerning the characteristic 
of the underlying field is unnecessary for all the results 
in \cite[Section 8]{J}; see \cite{J3} and \cite{S}. 
In particular, we have

\begin{thm}\cite[Theorem 8.13]{J}\cite[Proposition 8.10]{J} 
\label{thm:J} Let $S^d$ be the $d$-dimensional sphere with $d > 1$ and 
$\K$ an arbitrary field. Then  
the Auslander-Reiten quiver of the category $\text{\em D}^c(C^*(S^d; \K))$
 consists of $d-1$ components, each isomorphic to 
the translation quiver ${\mathbb Z}A_\infty$; see \cite[5.6]{H}. 
The component containing $Z_0\cong C^*(S^d;\K)$ is of the form
$$
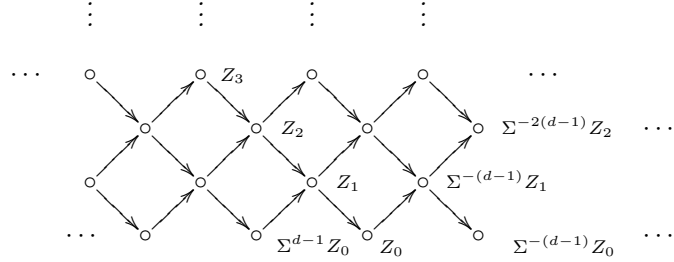
\objectmargin={0.5pt}
\xymatrix@C15pt@R15pt{
&  \vdots & {}\ar@{}[rd]^(1.2){Z_3} & \vdots & & \vdots & & \vdots \\
\cdots & \circ \ar[rd] &   & \circ \ar[rd]^(1.3){Z_2} & & \circ
                       \ar[rd]
        & & \circ  \ar[rd]^(1.3){\Sigma^{-2(d-1)}Z_2} & & \cdots \\
&   &\circ\ar[rd]\ar[ru]  &   & \circ \ar[rd]^(1.3){Z_1} \ar[ru] & 
               & \circ\ar[rd]^(1.3){\Sigma^{-(d-1)}Z_1} \ar[ru] & & \circ&
   & & \cdots \\
& \circ\ar[rd]\ar[ru]     &  & \circ \ar[rd]\ar[ru]&  & \circ\ar[rd]_(0.82){\Sigma^{d-1}Z_0}\ar[ru]&  & \circ
 \ar[rd]^(1.5){\Sigma^{-(d-1)}Z_0}\ar[ru] & & \\ 
& \!\!\!\! \cdots  &\circ \ar[ru] &   & \circ \ar[ur] & & \circ\ar[ru]_(0.08){Z_0}
 & & \circ  & & & \cdots 
}
$$
Moreover, the cohomology of the indecomposable object $\Sigma^{-l}Z_m$
 has the form  
$$
H^i(\Sigma^{-l}Z_m)\cong 
\left\{
\begin{array}{l}
\K  \ \ \text{for} \ i = -m(d-1)+l \ \text{and} \ d+l,  \\
0  \ \ \text{otherwise}.
\end{array}
\right.
$$ 
\end{thm}

In what follows, we call an indecomposable object in $\D^c(C^*(X; \K))$ 
a molecule. 

\begin{rem}\label{rem:Krull-Remak-Schmidt}
Let $A$ be a DG algebra with $\dim H(A) < \infty$. Then $\D^c(A)$ is 
a Krull-Remak-Schmidt category; that is, each object decomposes uniquely
 into indecomposable objects; see \cite[Proposition 2.4]{J3}. 
\end{rem}

\begin{rem}
\label{rem:classification}
The latter half of Theorem \ref{thm:J} implies that molecules 
in $\D^c(C^*(S^d; \K))$ are characterized by their cohomology. 
Moreover, those objects are also classified by the {\it amplitude} 
of their cohomology of the objects, up to shifts. Here 
the amplitude of a DG module $M$, denoted $\text{amp}M$, is defined by 
$$
\text{amp}M:=\sup \{ i \in {\mathbb Z} \ | \ M^i \neq 0 \}-
\inf\{ i \in {\mathbb Z} \ | \ M^i \neq 0\}.
$$ 
\end{rem}

The cohomology of $\Sigma^{-(d-1)}Z_1$ is isomorphic to 
$H^*(S^{2d-1}; \K)$ as a graded vector space and that 
there is an irreducible map $\Sigma^{-(d-1)}Z_1 \to Z_0$ that  induces 
$H^*(S^d; \K)=H^*(Z_0)\to H^*(\Sigma^{-(d-1)}Z_1)=H^*(S^{2d-1}; \K)$ 
a morphism of $H^*(S^d; \K)$-modules. 
Thus one might expect that the topological 
realizability of the morphism in the quiver 
is related to the Hopf invariant 
$H : \pi_{2d-1}(S^d)\to {\mathbb Z}$. We define realizability as follows.

\begin{defn}
An object $M$ in the category 
$\D^c(C^*(X; \K))$ is {\it realizable} by an object $f: Y \to X$ in 
${\mathcal TOP}_X$ if $M$ is isomorphic  
to the cochain complex 
 $C^*(Y; \K)$ endowed with the $C^*(X; \K)$-module structure via the map 
$f^* : C^*(X; \K) \to  C^*(Y; \K)$; that is, $M\cong C^*(Y; \K)^f$ in $\D^c(C^*(X; \K))$.   
\end{defn}

We establish the following proposition.

\begin{prop} 
\label{prop:spheres}
Let $\phi : S^{2d-1} \to S^d$ be a map.   
The DG module $C^*(S^{2d-1}; \K)^\phi$ over $C^*(S^d; \K)$ is in $\DD^c(C^*(S^d; \K))$ if and only if 
$H(\phi)_{\K}$ is nonzero,  where $H(-)_{\K}$ denotes the composite of the
 Hopf invariant with the reduction ${\mathbb Z} \to {\mathbb Z}\otimes \K$. 
In that case, the induced map 
$\phi^* : C^*(S^d; \K) \to C^*(S^{2d-1}; \K)$ coincides with the
 irreducible map $Z_0 \to \Sigma^{-(d-1)}Z_1$ up to scalar multiple. 
\end{prop}

Since the $0$th cohomology of a space is non-zero and the negative part
of the cohomology is zero, only indecomposable objects of the form $\Sigma^{-m(d-1)}Z_m$
($m\geq 0$) may be realizable; see the beginning of the proof of 
Theorem \ref{thm:realization}. 
Observe that the objects $\Sigma^{-m(d-1)}Z_m$ 
lie in the line connecting $Z_0$ and $\Sigma^{-(d-1)}Z_1$.
However, 
the following proposition states that most of molecules in 
$\D^c(C^*(X; \K))$ are {\it not} realizable by finite CW complexes.

\begin{thm}
\label{thm:realization} 
Suppose that the characteristic of the underlying field is greater than 
$2$ or zero. 
A molecule of the form $\Sigma^{-i}Z_l$  in 
$\text{\em D}^c(C^*(S^d; \K))$ is realizable by a finite CW complex 
if and only if $i=d-1$, $l=1$ and $d$ is even, 
or $i=0$ and $l=0$.  
\end{thm}

The rest of this paper is organized as follows.  
Section 3 contains a brief introduction to semifree resolutions. We also recall 
some results on the levels which we use later on.   
Section 4 is devoted to
proving Theorems \ref{thm:main}, while   
Proposition \ref{prop:pile} and Theorem \ref{thm:ex-level}
are proved in Section 5.
In Section 6, we prove Proposition
\ref{prop:spheres} and Theorem \ref{thm:realization}. 
The explicit computations 
of levels described in Propositions \ref{prop:bundle} and
\ref{prop:variation} are made in Section 7.

We conclude this section with comments on our work.     

\begin{rem}
Let $X$ be a simply-connected space whose
cohomology with coefficients in a field $\K$ is a Poincar\'e duality algebra. 
The Auslander-Reiten quiver of $\D^c(C^*(X; \K))$ 
then graphically depicts  irreducible morphisms 
and molecules in the
full subcategory. 
Even if a molecule in $\D^c(C^*(X; \K))$ is not realizable,  
it may be needed to construct $C^*(Y; \K)$ for a space $Y$ over $X$ 
as a $C^*(X; \K)$-module. 
In fact, it follows from the proofs of Propositions \ref{prop:bundle} and
\ref{prop:variation} that some molecules are retracts of 
the $C^*(S^4; \K)$-modules $C^*(E_f; \K)$ and $C^*(E_g; \K)$,  even though 
they are not realizable; see also Example \ref{ex:ex}.  
\end{rem}  
  
\begin{rem} 
A CW complex $Z$ is built out disks, which are called cells,
by iterated attachment of them. 
It is well-known that the dual to the cellular chain complex of a CW complex 
$Z$ is quasi-isomorphic to the singular cochain complex $C^*(Z; \K)$. 
Thus $C^*(Z; \K)$ is also regarded as 
`a set of cells' and hence it seems a creature in some sense. 
When we describe images by the functor $C^*(- ; \K)$ in terms of representation theory,  
we may need objects in $\D^c(C^*(X; \K))$ which are not necessarily realizable.  
Therefore one might regard such an object as structurally smaller
than a cell. This is the reason why we give indecomposable objects in
$\D^c(C^*(X; \K))$ the name `molecules'. 
\end{rem}

\section{Semifree  resolutions and the levels}

We begin by recalling the definition of the semifree resolution. 

Let $A$ be  a DG algebra over $\K$. 
\begin{defn} \cite[4.1]{ABIM}\cite{FHT_G}\cite[\S 6]{F-H-T} 
A  {\it  semifree filtration} of a DG $A$-module $M$ is a family $\{F^n\}_{n\in {\mathbb Z}}$ of
DG submodules of $M$ satisfying the condition:
$F^{-1}=0$, $F^n \subset F^{n+1}$, $\cup_{n\geq 0}F^n =M$ and 
$F^n/F^{n-1}$ is isomorphic to a direct sum of shifts of $A$. 
A DG $A$-module $M$ admitting a semifree filtration is called {\it semifree}. 
We say that the filtration $\{F^n\}_{n\in {\mathbb Z}}$ has 
{\it class at  most} $l$
if $F^l=M$ for some integer $l$. 
Moreover  $\{F^n\}_{n\in {\mathbb Z}}$ is called {\it finite} if each subquotient is finitely generated.   
\end{defn}

Let $M$ be a DG $A$-module. We say that a quasi-isomorphism  of $A$-modules 
$F \stackrel{\simeq}{\to} M$ is 
a {\it  semifree resolution} of $M$ if $F$ is semifree. For example, the bar resolution 
$B(M;A;A)$ of $M$ is $A$-semifree, and its canonical 
augmentation $\e : B(M;A;A) \stackrel{\simeq}{\to} M$ is therefore a semifree resolution of $M$. 

Let $N$ be a left DG $A$-module. We observe that the
left derived functor $- \otimes^\text{L}_AN$ is defined by  
$M\otimes^\text{L}_AN:=F\otimes_AN$ for any right DG module $M$ over
$A$, where $F \stackrel{\simeq}{\to} M$ 
is a semifree resolution of $M$.  
We see that by definition $H^*(M\otimes_A^LN)$ is exactly $\text{Tor}_A(M, N)$. 
 
The following result is useful for computing the $A$-level of an object in $\D(A)$. 

\begin{thm}\cite[Theorem 4.2]{ABIM}
\label{thm:ABIM-I}
Let $M$ be a DG module over a DG algebra $A$ and $l$ a non-negative
integer. Then 
$\text{\em level}_{\text{\em D}(A)}^A(M)\leq l$ if and only if 
$M$ is a retract in $\text{\em D}(A)$ of some DG module admitting a finite
semifree filtration of class at most $l-1$. 
\end{thm}

In order to study Auslander-Reiten triangles, in \cite{J2}, 
J{\o}rgensen introduced  
the function $\varphi : \D(A) \to {\mathbb Z}\cup \{\infty\}$  
defined by 
$$
\varphi (M):=\dim H^*(M\otimes_A^\text{L}\K). 
$$
This yields a criterion for a given object in $\D(A)$ to be compact.

\begin{prop}\cite[Theorem 4.8]{ABIM}\cite[Proposition 2.3]{F-J}\cite[Theorem 5.3]{Keller} 
\label{prop:compact} Let $A$ be a simply-connected DG algebra.  
An object $M$ in $\text{\em D}(A)$ is compact if and only if $\varphi (M) < \infty$.  
In that case $\text{\em level}_{\DD(A)}^A(M) < \infty$. 

In particular, 
for a map $\phi : Y \to X$ from a connected space $Y$ to a simply-connected space
 $X$,  if the total dimension of the cohomology of the homotopy fibre of the map 
$\phi$ is finite,  then 
$C^*(Y; \K)$ is in ${\DD}^c(C^*(X; \K))$ and hence 
$\text{\em level}_{X}(Y) < \infty$. 
\end{prop} 

\begin{rem}
\label{rem:fibre}
Let $F_\phi$ be the homotopy fibre of a map $\phi : Y \to X$. 
The latter half of Proposition  \ref{prop:compact} follows from the fact that 
$H^*(F_\phi; \K)\cong \text{Tor}_{C^*(X; \K)}(C^*(Y; \K), \K)\cong H^*(C^*(Y; \K)\otimes_{C^*(X; \K)}^\text{L}\K)$ as a graded vector space; see \cite{Smith}\cite[Theorem 7.5]{F-H-T}. 
\end{rem}

We conclude this section with a result due to Schmidt, about the levels of  
molecules in $\D^c(C^*(S^d; \K))$, which is used in the proof of Theorem \ref{thm:ex-level}. 
 
 \begin{prop} \cite[Proposition 6.6]{S}
\label{prop:Z} \ Let $Z_i$ be the molecule in  $\text{\em D}^c(C^*(S^d; \K))$
 described in Theorem \ref{thm:J}. Then  
$\text{\em level}_{\text{\em D}(C^*(S^d; \K))}(Z_i)= i+1$. 
\end{prop}

\section{Proof of Theorem \ref{thm:main}}
In what follows, we write $C^*( \ )$ and $H^*( \ )$ for $C^*( \ ; \K)$ and $H^*(\ ; \K)$, respectively 
if the coefficients are clear from the context. 

Let $X$ be a simply-connected formal space and 
$m_X : TV_X \stackrel{\simeq}{\to} C^*(X; \K)$ be a
minimal model. 
We then have the following equivalences 
of triangulated categories; see \cite[Proposition 4.2]{KM},  
$$
\xymatrix@C30pt@R15pt{
\D(C^*(X; \K)) \ar[r]_(0.55){\simeq}^(0.55){m_X^*} & \D(TV_X) 
 \ar[r]^{- \otimes_{TV_X}^\text{L}H^*(X; \K)}_(0.45){\simeq} & 
\D(H^*(X; \K)),  
}
$$
where $m_X^*$ is the pullback functor; that is, for a $C^*(X; \K)$-module $M$, 
$m_X^*M$ is defined to be the module $M$ endowed with the $TV_X$-module structure via $m_X$. 
We denote by $F_X$ the composite of the functors: 
$F_X = - \otimes_{TV_X}^\text{L}H^*(X; \K) \circ m_X^*$.  
Observe that the functor $F_X$ leaves the cohomology of an object
unchanged; see \cite[Proposition 6.7]{F-H-T} for example. 

\begin{lem} 
\label{lem:key} 
Under the same hypothesis as in Theorem \ref{thm:main}, the differential
 graded module 
$F_X(C^*(E\times_B X; \K))$
 is isomorphic to $H^*(E; \K)\otimes^\text{\em L}_{H^*(B; \K)}H^*(X; \K)$ 
in the category $\text{\em D}(H^*(X; \K))$.  
\end{lem}

\begin{proof} We use the same notation as in Definition \ref{defn:formalizable}.  
Let $H : TV_B\wedge I \to C^*(E)$ and 
$K : TV_B\wedge I \to C^*(X)$ be homotopies from $q^*\circ m_B$ to 
$m_E\circ \widetilde{q}$ and from $f^*\circ m_B$ to 
$m_E\circ \widetilde{f}$, respectively. Here $TV_B\wedge I$ denotes the
 cylinder object due to Baues and Lemaire \cite{FHT} in the category of
 DG algebras; see Appendix. 
The homotopies $H$ and $K$ make $C^*(E)$ and $C^*(X)$ into a right 
$TV_B\wedge I$-module and a left $TV_B\wedge I$-module, respectively, so there exists a right 
$C^*(X)$-module of the form 
$C^*(E)\otimes^\text{L}_{TV_B\wedge I}C^*(X)$. 
Then there exists a sequence of quasi-isomorphisms of $TV_X$-modules 
$$
\xymatrix@C25pt@R16pt{
C^*(E\times_BX) & C^*(E)\otimes^\text{L}_{C^*(B)}C^*(X)
 \ar[l]^(0.55){\simeq}_(0.55){EM} &  C^*(E)\otimes^\text{L}_{TV_B}C^*(X)
 \ar[d]^{1\otimes_{\e_0}1}_{\simeq} \ar[l]_{1\otimes_{m_B}1}^{\simeq} \\
TV_E\otimes^\text{L}_{TV_B}TV_X \ar[r]_(0.45){\simeq}^(0.45){m_E\otimes_1m_X} &
C^*(E)\otimes^\text{L}_{TV_B}C^*(X)
 \ar[r]^{1\otimes_{\e_1}1}_{\simeq}  
 & C^*(E)\otimes^\text{L}_{TV_B\wedge I}C^*(X), 
}
$$
where $EM$ denotes the Eilenberg-Moore map; 
see \cite[Theorem 3.2]{Smith}. 
Therefore we see that  
$m_X^*(C^*(E\times_BX))$ is
 isomorphic to $TV_E\otimes^\text{L}_{TV_B}TV_X$ in $\D(TV_X)$. 
By considering the bar resolution of $TV_E$ as a $TV_B$-module, we see that, as 
objects in  $\D(H^*(X))$,  
$(TV_E\otimes^\text{L}_{TV_B}TV_X)\otimes^\text{L}_{TV_X}H^*(X)$
is isomorphic to $TV_E\otimes^\text{L}_{TV_B}H^*(X)$. 
Then a sequence of quasi-isomorphisms similar to that above connects 
$TV_E\otimes^\text{L}_{TV_B}H^*(X)$ with 
$H^*(E)\otimes^\text{L}_{H^*(B)}H^*(X)$ 
in $\D(H^*(X))$. In fact we have quasi-isomorphisms
$$
\xymatrix@C25pt@R16pt{
TV_E\otimes^\text{L}_{TV_B}H^*(X) \ar[r]_{\simeq}^{\phi_E\otimes 1} & 
H^*(E)\otimes^\text{L}_{TV_B}H^*(X) \ar[r]_\simeq^{1\otimes_{\e_0}1}
 & H^*(E)\otimes^\text{L}_{TV_B\wedge I}H^*(X) \\
& H^*(E)\otimes^\text{L}_{H^*(B)}H^*(X) 
    & H^*(E)\otimes_{TV_B}^\text{L}H^*(X) 
 \ar[l]_{\simeq}^{1\otimes_{\phi_B}1} \ar[u]_{1\otimes_{\e_1}1}^{\simeq}.
}
$$
This completes the proof.
\end{proof}

\medskip
\noindent
{\it Proof of Theorem \ref{thm:main}.} We see that in $\D(H^*(X))$ 
$$F_X C^*(X) = (m_X^*C^*(X))\otimes_{TV_X}^LH^*(X)= 
TV_X\otimes^L_{TV_X}H^*(X) =H^*(X).
$$ 
Then the result \cite[Proposition 3.4 (1)]{ABIM} allows us to deduce that 
$\text{level}_{\D(C^*(X; \K))}(M)=
\text{level}_{\D(H^*(X; \K))}(F_XM)$
for any object $M$ in  $\D(C^*(X, \K))$. 
By virtue of Lemma \ref{lem:key}, we have the result.  
\hfill\qed

\medskip
We recall a fundamental property of an object laying in the thickening of $\D(A)$. 
The result follows from the fact that a triangle induces a long exact sequence in homology. 

\begin{lem}\label{lem:fund}
Let $A$ be a DG algebra, $M$  a DG $A$-module and $n$ a positive integer. 
Suppose that $\dim H(A) < \infty$. Then $\dim H(M) <\infty$ 
for any object $M \in {\tt thick}_{\DD(A)}^n(A)$. 
\end{lem}

\begin{ex}
\label{ex:S7} Let $\nu : S^7 \to S^4$ be the Hopf map and $E_{\nu}$ the
 pullback of $\nu : S^7 \to S^4$ over itself, giving rise to a fibration 
$S^3 \to E_{\nu} \to S^7$. 
We prove now that 
$$
\text{level}_{S^7}(E_{\nu}) 
\neq \text{level}_{\D(H^*(S^7; \K))}(H^*(S^7;\K)\otimes_{H^*(S^4;\K)}^\text{\em L}H^*(S^7; \K)).  
\eqnlabel{add-0}
$$
Indeed, there is a Koszul resolution of the form 
$$
(\Gamma[w]\otimes \wedge(s^{-1}x_4)\otimes H^*(S^4; \K), \delta) \to \K\to 0
$$
with $\delta(s^{-1}x_4)=x_4$ and $\delta(\omega) = s^{-1}x_4\otimes x_4$, where $x_4$ denotes the generator of 
$H^*(S^4; \K)$, and $\Gamma$ the divided powers algebra functor; see \cite[Proposition1.2]{K1}. 
This gives rise to a semifree resolution 
$$
H^*(S^7; \K)\otimes \Gamma[w]\otimes \wedge(s^{-1}x_4)\otimes H^*(S^4; \K)
\to H^*(S^7; \K) \to 0
$$
of $H^*(S^7; \K)$ as an $H^*(S^4; \K)$-module. 
Thus we have 
$$
M:=H^*(S^7;\K)\otimes_{H^*(S^4;\K)}^\text{\em L}H^*(S^7; \K)
=(H^*(S^7; \K)\otimes \Gamma[w]\otimes \wedge(s^{-1}x_4)\otimes H^*(S^7;
 \K), 0).
$$ 
Since $\dim H(M)= \infty$, it follows from Lemma \ref{lem:fund} that $M$ is not in the thickening 
$\text{{\tt thick}}^n_{\D(H^*(S^7; \K))}(H^*(S^7; \K))$ 
for any $n\geq 0$. 
This implies that the right hand side of (4.1) is infinite. 

 On the other hand, by Proposition \ref{prop:compact}, we see that 
$\text{level}_{S^7}(E_{\nu}) < \infty$ 
because the dimension of the cohomology of the fibre $S^3$ is finite. 
We refer the reader to Example \ref{ex:non-formal} for the explicit
 calculation of the level of $E_\nu$. 
\end{ex}

\section{Proofs of Proposition \ref{prop:pile} and Theorem \ref{thm:ex-level} }

In this section, we work in rational homotopy theory and use 
Sullivan models for spaces and fibrations extensively.  
For a thorough introduction to these models, we refer the reader to the book \cite{F-H-T}.

As mentioned in the Introduction, Theorem \ref{thm:ex-level} is deduced from 
Proposition \ref{prop:pile}. The proof of the proposition is given first. 

\medskip
\noindent
{\it Proof of Proposition \ref{prop:pile}.} Let $Y_0$ be the space
$B\times (\displaystyle{\times_{i=1}^sS^{2n_i+1}})$ 
and $\Lambda V_B$ a minimal model for $B$. 
Then the Sullivan model for the fibration 
$S^{2m_i+1} \to Y_i \stackrel{\pi_c}{\to} Y_{i-1}$ has the form 
$\wedge V_{i -1}\to \wedge (x_i)\otimes \wedge V_{i-1}=\wedge V_i$, where 
$\wedge V_0 = \wedge V_B\otimes \wedge (y_{01}, ..., y_{0s})$ with 
$d(y_{0i})= 0$. 
Since the DG algebras $C^*(B; \Q)$ and $\wedge V_B$ are connected with
quasi-isomorphisms, it follows from \cite[Proposition 4.2]{KM} and \cite[Lemma 2.4]{ABIM} that  
$\text{level}_{B}Y_c =
\text{level}_{\text{D}(\wedge V_B)}\wedge V_c$. 
 
Define a filtration $\{F_l\}_{0\leq l\leq c}$ 
of the $\wedge V_B$-module $\wedge V_c$ by
$$
F_l = \Lambda V_B\otimes \Q\{y_{01}^{\e_{01}}\cdots y_{0s}^{\e_{01}}x_1^{\e_1}
\cdots x_l^{\e_l} \ | \ \e_{0i} \ \text{and} \ \e_j \ \text{are} \
0 \ \text{or} \ 1 \}.  
$$ 
It is immediate that $F_l/F_{l-1}$ is a finitely generated free $\wedge V_B$-module for each 
$l \geq 0$. Then it follows that $\{F_l\}_{0\leq l\leq c}$ is a finite semifree
filtration of class at most $c$. By virtue of Theorem \ref{thm:ABIM-I}, we have 
$\text{level}_{\text{D}(\wedge V_B)}\wedge V_c \leq c+1$.
\hfill \qed

\medskip
We now establish a weaker version of Theorem \ref{thm:ex-level}.  

\begin{lem} \label{lem:ex-level}
For any positive integer $l$, 
 there exists an object $P_l \to S^d$
 in $\mathcal{TOP}_{S^d}$ such that 
$$
\text{\em level}_{S^d}(P_l)\geq l. 
$$ 
\end{lem}

\begin{proof}
In the case where $l=1$, the sphere $S^d$ is the space we desire.  
In what follows, we assume that $l\geq 2$. 
Let $m$ be an integer sufficiently larger than $ld$. 

Assume that $d$ is even. 
We have a minimal model 
$B=(\wedge (x,  \xi), \delta)$ for $S^d$ with $\delta(\xi) = x^2$,
where $\deg x = d$.  Consider a Koszul-Sullivan extension of the form 
$$
B \to (\wedge(x, \xi, \rho, w_0, ..., w_{l-1}), D)=:M_{l+1}
$$
for which the differential $D$ is defined by 
$$
D(\rho)=x,  D(w_0)=0 \ \text{and} \   
D(w_i)=(\rho x-\xi)w_{i-1}
$$ 
for $i\geq 1$, where 
$\deg w_i = i(2d-1)+(2m-1)-i$. 
Let $\pi : P_{l+1} \to S^d$ be the pullback of the fibration 
$|M_{l+1}| \to |B|=S^d_{{\mathbb Q}}$, which is the spatial realization of
the extension, by the localizing map $S^d \to S^d_{{\mathbb Q}}$; see
\cite[Proposition 7.9]{F-H-T}. 
Since $M_{l+1}$ is a semifree $B$-module, it follows that  
$H^*(M_{l+1}\otimes_B^L{\mathbb Q})=H^*(M_{l+1}\otimes_B{\mathbb Q})=
H^*(\wedge (\rho, w_0, w_1, ..., w_{l-1}), \overline{D})$. 
The cochain complex  $M_{l+1}\otimes_B{\mathbb Q}$ is generated by elements with odd degree so that 
its homology is of finite dimension. 
It follows from Proposition \ref{prop:compact} that $C^*(P_{l+1}; \Q)$ is in $\D^c(C^*(S^d; \Q))$.


By using the manner in \cite[Section 7]{K-M-N} for computing the 
homology of a DG algebra (or by the direct calculation), we have elements $1$,
$\xi$, $w_0$ and $(\rho x-\xi)w_{l-1}$, which form a basis of $H^*(M_{l+1})$ of 
degree less than or equal to $l(2d-1)+(2m-1)-(l-1)$. Let $Z$ be an 
indecomposable direct summand (a molecule) of $C^*(P_{l+1}; \Q)$ in $\D^c(C^*(S^d; \Q))$
containing a cocycle of degree zero; see Remark
 \ref{rem:Krull-Remak-Schmidt}. 
By virtue of Theorem \ref{thm:J}, 
we see that 
$Z=\Sigma^{-k(d-1)}Z_k$ for some $k\geq 0$; 
see Remark \ref{rem:classification}. 
Suppose that $Z$ contains 
a representative of $w_0$, $(\rho x-\xi)w_{l-1}$ or a cohomology class 
of degree greater than $l(2d-1)+(2m-1)-(l-1)$.  
Theorem \ref{thm:J} implies that $H^i(Z)=\Q$ if and only
if $i=(k+1)d-k$ or $i=0$. It follows that $(k+1)d-k\geq 2m-1 \geq
2ld-1$ and hence $k\geq 2l-1\geq l$.  Then Proposition \ref{prop:Z} allows us
to conclude that $\text{level}_{S^d}(P_{l+1})\geq
l+1$.  

Suppose that $Z$ contains a representative of the element $\xi$. 
By Theorem \ref{thm:J}, we see that $Z=\Sigma^{-(d-1)}Z_1$. 
In that case, let $Z'$ be a molecule of $C^*(P_{l+1}; \Q)$ containing a representative of $w_0$. 
Observe that $Z'\neq Z$. If $Z'$ contains a representative of the
element $(\rho x-\xi)w_{l-1}$, then
$Z'=\Sigma^{-(2m-1)}\Sigma^{-(2l-1)(d-1)}Z_{2l-1}$ since 
$\dim H^*(Z')=2$ and the amplitude of $Z'$ should be $2ld-2l+1$. 
If $Z'$ contains a representative of the cohomology class of degree greater than $l(2d-1)+(2m-1)-(l-1)$, 
then $Z'=\Sigma^{-(2m-1)}\Sigma^{-(2l-1)(d-1)}Z_k$ 
for some $k\geq 2l-1$. 
Proposition \ref{prop:Z} yields that 
$\text{level}_{S^d}(P_{l+1})\geq 2l$.   
 
Suppose that $d$ is odd. We have a Koszul-Sullivan extension of the form 
$$
(\wedge(x), 0) \to (\wedge(x, w_0, w_1, ..., w_{l-1}), D)=:N_l
$$
for which the differential $D$ is defined by $D(x)=D(w_0)=0$ and 
$D(w_i)=xw_{i-1}$ for $i\geq 1$, where $\deg x=d$ and $\deg w_0=2m-1$.
We assume that the integer $m$ is sufficiently larger than $ld$.  Observe that 
$\deg w_i = id+(2m-1)-i$. Let  
$\pi : P_l \to S^d$ be the pullback of the fibration 
$|N_l| \to |(\wedge (x), 0)|=S^d_{{\mathbb Q}}$  
by the localizing map $S^d \to S^d_{{\mathbb Q}}$. The same argument as
above works again to show that 
$\text{level}_{S^d}(P_l)\geq l$. This
completes the proof. 
\end{proof}

\medskip
\noindent
{\it Proof of Theorem \ref{thm:ex-level}.} 
Let $P_l \to S^d$ be the fibration constructed 
in the proof of Lemma \ref{lem:ex-level}.  
We have a sequence of fibrations 
$$
S^{|\rho|} \to Y_1 \stackrel{\pi_1}{\longrightarrow} S^d  
 \times S^{|w_0|},   \ 
S^{|w_1|}  \to Y_2 \stackrel{\pi_2}{\longrightarrow} Y_1, \ ..., \
S^{|w_{l-1}|} \to Y_l \stackrel{\pi_l}{\longrightarrow} Y_{l-1}  
$$
in which $Y_l=P_{l+1}$ if $d$ is even, where $|w|$ denotes the degree of an element $w$. 
If $d$ is odd, we have a sequence
of fibrations 
$$
S^{|w_1|} \longrightarrow Y_1 \stackrel{\pi_1}{\longrightarrow} S^d \times S^{|w_0|},  \ 
S^{|w_2|}  \to Y_2 \stackrel{\pi_2}{\longrightarrow} Y_1, \ ..., \ 
S^{|w_{l-1}|} \to Y_{l-1} \stackrel{\pi_{l-1}}{\longrightarrow} Y_{l-2}  
$$
in which $Y_{l-1}=P_{l}$.  
Observe that the integers $|\rho|$ and $|w_i|$ are odd. 
It follows from Proposition \ref{prop:pile} that 
$\text{level}_{S^d}P_l \leq l$. 
By  combining the result with Lemma \ref{lem:ex-level}, the proof is now completed.    
\hfil \qed

\section{Realization of molecules in $\D^c(C^*(S^d; \K))$}

We recall briefly the Hopf invariant. 
Let $\phi : S^{2d-1} \to S^d$ be a map. 
Choose generators $[x_{2d-1}]\in H^{2d-1}(S^{2d-1}; {\mathbb Z})$ and $[x_d]\in
H^d(S^d; {\mathbb Z})$. 
Let $\rho$ be an element of $C^*(S^{2d-1}; {\mathbb Z})$ such that $\phi^*(x_d)=
d\rho$. Since $[x_d]^2=0$ in $H^*(S^d; {\mathbb Z})$, there exists an element 
$\xi$ of $C^*(S^d; {\mathbb Z})$ such that 
$d\xi = x_d^2$. We then have a cocycle of
the form $\rho \phi^*(x_d)-\phi^*(\xi)$ in $C^{2d-1}(S^{2d-1})$. 
The Hopf invariant $H(\phi) \in {\mathbb Z}$ is defined by the equality 
$$
[\rho \phi^*(x_d)- \phi^*\xi] = H(\phi)[x_{2d-1}]. 
$$  

\begin{rem} 
If $d$ is odd, then $H(\phi)$ is always zero.  
\end{rem}

We prove Proposition \ref{prop:spheres} by using 
Proposition \ref{prop:compact}. To this end, we need to consider whether 
the cohomology $H^*(F_\phi; \K)$ is of finite dimension, 
where $F_\phi$ denotes the homotopy fibre of $\phi : S^{2d-1}\to S^d$. 
Observe that $F_\phi$ fits into the pullback diagram ${\mathcal F}'$ : 
$$
\xymatrix@C25pt@R15pt{
\Omega S^d \ar[d] \ar@{=}[r] &  \Omega S^d \ar[d] \\
F_\phi \ar[r] \ar[d] & PS^d \ar[d]^{\pi} \\
S^{2d-1}  \ar[r]_{\phi}  & S^d   .
}
$$
Here $\Omega S^d \to PS^d \stackrel{\pi}{\to} S^d$ is the path-loop fibration. 
The pullback diagram gives rise to the Eilenberg-Moore spectral sequence
$\{E_r^{*,*}, d_r\}$ converging to   
$H^*(F_\phi; \K)$ with 
$$
E_2^{*,*}\cong \text{Tor}^{*,*}_{H^*(S^d; \K)}(H^*(S^{2d-1}; \K), \K).
$$ 
The Koszul resolution of $\K$ as an 
$H^*(S^d; \K)$-module allows us to compute the $E_2$-term. It turns out 
that   
$$
E_2^{*,*}\cong
\left\{
\begin{array}{ll}
 H^*(S^{2d-1};\K)\otimes \wedge(s^{-1}x_d)\otimes \Gamma[\tau]  & 
\text{if} \ d \ \text{is even},   \\
 H^*(S^{2d-1};\K)\otimes \Gamma[s^{-1}x_d]  &  \text{if} \ d \ \text{is
 odd}, 
\end{array} 
\right. 
$$ 
where $\text{bideg} \ s^{-1}x_d=(-1, d)$ and $\text{bideg} \ \tau=(-2, 2d)$;
see \cite[Lemma 3.1]{Smith2} and also \cite[Proposition 1.2]{K1}. 

We relate the Hopf invariant with a differential of the Eilenberg-Moore
spectral sequence (EMSS). 

Recall that the Eilenberg-Moore map induces an isomorphism 
from the homology of the 
bar complex 
$(B(C^*(S^{2d-1}; \K), C^*(S^d; \K), \K), \delta_1+
\delta_2)
$ 
to $H^*(F_\phi; \K)$. Here $\delta_1$ denotes the part of the differential coming from the multiplication 
of the algebra and its action on the module, which decreases bar-length, while $\delta_2$ is induced by the differentials 
of the algebra and module and does not change bar-length. 
By the definitions of differentials $\delta_1$ and $\delta_2$, we see
that  
\begin{eqnarray*}
\delta_1([x_d|x_d])&=&(-1)^d \phi^*(x_d)[x_d] + (-1)^d(-1)^{d+1}[x_d^2]  \\
 &=& \delta_2((-1)^d\rho[x_d] + 1[\xi]),  \\
\delta_1((-1)^d\rho[x_d] + 1[\xi]) &=& (-1)^d\{(-1)^{d-1}\rho \phi^*x_d\} +
 \phi^*\xi \\
&=& - (\rho \phi^*(x_d)- \phi^*\xi). 
\end{eqnarray*} 
It follows from \cite[Lemma 2.1]{K-S} 
that $d_2([x_d|x_d])=H(\phi)_{\K}x_{2d-1}$ in 
the $E_2$-term of the EMSS. 

We denote by $\text{Tor}_{H^*(S^d; \K)}(H^*(S^{2d-1}; \K), \K)_{\text{bar}}$ the torsion product as 
computed by the bar complex, which is necessarily isomorphic to the torsion product computed 
by the Koszul resolution. 

By the same argument as in \cite[Lemma 1.5]{K1}, we have: 
\begin{lem}
\label{lem:E_2} 
The element $[x_d|x_d]$ in 
$\text{\em Tor}_{H^*(S^d; \K)}(H^*(S^{2d-1}; \K), \K)_{\text{\em bar}}$ 
coincides with the element $\tau \in \Gamma[\tau]$ up to isomorphism 
if $d$ is even and with the element 
$\gamma_2(s^{-1}x_d) \in \Gamma[s^{-1}x_d]$ if $d$ is odd. Thus one has 
$d_2(\tau)=H(\phi)_\K x_{2d-1}$ if $d$ is even and 
$d_2(\gamma_2(s^{-1}x_d))=H(\phi)_\K x_{2d-1}=0$ if $d$ is odd. 
\end{lem}



\medskip
\noindent
{\it Proof of Proposition \ref{prop:spheres}.}
Let $\{\widetilde{E}_r^{*,*}, \widetilde{d}_r\}$ be the EMSS 
converging to $H^*(\Omega S^d; \K)$. We see that 
$$
\widetilde{E}_2^{*,*}\cong
\left\{
\begin{array}{ll}
 \wedge(s^{-1}x_d)\otimes \Gamma[\tau]  & 
\text{if} \ d \ \text{is even},   \\
 \Gamma[s^{-1}x_d]  &  \text{if} \ d \ \text{is
 odd}, 
\end{array} 
\right. 
$$ 
where $\text{bideg} \ s^{-1}x_d = (-1, d)$ and 
$\text{bideg} \ \tau = (-2, 2d)$. 
The result \cite[Theorem III]{FHT}
 implies that  the EMSS for the fibre square
 ${\mathcal F}'$ is a right DG comodule over 
$\{\widetilde{E}_r^{*,*}, \widetilde{d}_r\}$; that is, there exists a
 comodule structure 
$\Delta : E_r^{*,*} \to E_r^{*,*}\otimes \widetilde{E}_r^{*,*} $ for any
 $r$
such that the diagram
$$
\xymatrix@C30pt@R20pt{
 E_r^{*,*}\otimes \widetilde{E}_r^{*,*} 
\ar[r]^{d_r\otimes 1 \pm 1\otimes \widetilde{d}_r} &  E_r^{*,*}\otimes
 \widetilde{E}_r^{*,*} \\
 E_r^{*,*} \ar[u]^{\Delta} \ar[r]_{d_r}&  E_r^{*,*} \ar[u]_{\Delta} 
}
$$
is commutative. 
Since the comultiplication of the bar construction induces the comodule structure, 
it follows that, in our case, 
$$
\Delta(x_{2d-1}^{\e}\gamma_i(\tau))=
 \sum_{0\leq l \leq i}x_{2d-1}^{\e}\gamma_{i-l}(\tau)\otimes \gamma_l(\tau),
$$
where $\e= 0$ or $1$. 
For dimensional reasons, we see that
 $\widetilde{d}_r=0$ for all $r$.  
 In fact if $i > j$, then we have 
$$
\text{t-deg} \ \gamma_i(\tau) +1 = 2i(d-1)+1 > (2j+1)(d-1) 
=\text{t-deg} \ s^{-1}x_d\gamma_j(\tau), 
\eqnlabel{add-0}
$$ 
where $\text{t-deg} \alpha$ denotes the total degree of an element $\alpha \in \widetilde{E}_2^{s,t}$, namely 
$\text{t-deg} \alpha = s+t$. 
This implies that $\widetilde{d}_r(\gamma_i(\tau))=0$ even if $d$ is even. 

Suppose that $H(\phi)_{\K}$ is nonzero. Then $d$ is even. The commutativity of the diagram above and 
Lemma \ref{lem:E_2} together allow 
us to deduce that $d_2(\gamma_i(\tau))=
 H(\phi)_{\K}x_{2d-1}\gamma_{i-1}(\tau)$, whence $H^*(F_\phi; \K)\cong H^*(S^{d-1}; \K)$. 
It follows then that the $C^*(S^d; \K)$-module $C^*(S^{2d-1}; \K)$ is 
in the category $\D^c(C^*(S^d; \K))$. 

We show that the converse holds. 
Assume that $C^*(S^{2d-1}; \K)$ is a compact object and $d$ is even.  
It follows from  Proposition \ref{prop:compact} that 
$\dim H^*(F_\phi; \K)<\infty$ so that there exists a non-trivial differential in the
 EMSS $\{E_r^{*,*}, d_r\}$. Let $\gamma_j(\tau) \in E_2^{*,*}$ 
be an element with the first non-trivial differential; that is, 
$d_s =0$ for $s < r$, 
$d_r(\gamma_j(\tau))\neq 0$ and $d_r(\gamma_i(\tau))=0$ for $i <j$.   
In view of the inequality (6.1),  we can write $d_r(\gamma_j(\tau))= \alpha x_{2d-1}\gamma_k(\tau)$, where
 $\alpha\neq 0$. We see that 
\begin{eqnarray*}
(d_r\otimes 1\pm 1\otimes\widetilde{d}_r) \Delta(\gamma_j(\tau))\! \!\!&=& \! \! \!
(d_r\otimes 1)(\sum_{0\leq t \leq j}\gamma_{t}(\tau)\otimes\gamma_{j-t}(\tau)) \\
&=&\sum_{0\leq t \leq j} d_r(\gamma_{t}(\tau))\otimes\gamma_{j-t}(\tau) \ = \ d_r(\gamma_j(\tau))\otimes 1. 
\end{eqnarray*}
Consider the commutative diagram mentioned above. We then have  
\begin{eqnarray*}
(d_r\otimes 1\pm 1\otimes\widetilde{d}_r) \Delta(\gamma_j(\tau)) \! \! \!&=&\! \! \!\Delta
 d_r(\gamma_j(\tau)) \\
 &=& \alpha(x_{2d-1}\otimes\gamma_k(\tau)+
\sum_{0< t \leq k} x_{2d-1}\gamma_t(\tau)\otimes \gamma_{k-t}(\tau)). 
\end{eqnarray*}
This amounts to requiring that $k = 0$. 
Thus we have  $d_r(\gamma_j(\tau))= \alpha x_{2d-1}$. 
The comparison between the total degrees allows us to deduce that 
$j(2(d-1))+1 = 2d-1$ and hence $j = 1$. For dimensional reasons, we have $r=2$. 
Lemma \ref{lem:E_2} yields that $\alpha = H(\phi)_{\K}$.  

In the case where $d$ is odd, the same argument works well to show the
 result.  
It follows from Theorem \ref{thm:J} that 
$C^*(S^{2d-1}; \K)\cong \Sigma^{-(d-1)}Z_1$ in $\D^c(C^*(S^d; \K))$; see
 also Remark \ref{rem:classification}.   

We show the latter half of the assertion. 
By considering the Auslander-Reiten quiver of $\D^c(C^*(S^d; \K))$, 
we see that there is an irreducible map from $C^*(S^d; \K)$ to 
$C^*(S^{2d-1}; \K)$. Observe that the map is non-trivial. 

Suppose that 
$\phi^* : C^*(S^d; \K) \to C^*(S^{2d-1}; \K)$ is trivial in 
$\D(C^*(S^d;\K))$. Then there exists a $C^*(S^d; \K)$-linear map 
$s :  C^*(S^d; \K) \to C^*(S^{2d-1}; \K)$ of degree $-1$ 
such that $\phi^* =sd + ds$. 
We see that $\phi^*(1) = sd(1) + ds(1) = 0$ because $d(1)=0$ and $\deg s = -1$. 
This yields that 
$\phi^*=0$ as a $C^*(S^d; \K)$-linear map. The definition of the Hopf invariant enables us to conclude that 
$H(\phi)_\K=0$; that is, $\phi^*\neq 0$ in $\D(C^*(S^d;\K))$ if $H(\phi)_\K\neq 0$. Moreover,  
$$
\text{Hom}_{\D(C^*(S^d;\K))}(C^*(S^d; \K), C^*(S^{2d-1}; \K))=
H^0( C^*(S^{2d-1}; \K))=\K.
$$
It follows that the map $\phi : S^{2d-1} \to S^d$ 
with non-trivial Hopf invariant induces an
irreducible map $\phi^*$ which coincides with the map $Z_0 \to
\Sigma^{-(d-1)}Z_1$ up to scalar multiple.   
\hfill\qed

\begin{rem}
\label{rem:EMSS-K-formal}
If the pair $(q,f)$ of maps in the fibre square ${\mathcal F}$ described 
 before Theorem \ref{thm:main} is relatively $\K$-formalizable, then
the EMSS sequence with coefficients in $\K$ for ${\mathcal F}$ collapses
at the $E_2$-term; see \cite[Proposition 3.2]{K}. 

Let $\phi : S^{2d-1} \to S^d$ be a map between spheres and 
$F_\phi$ the homotopy fibre of $\phi$. 
Then the proof of Proposition \ref{prop:spheres} yields that the EMSS
 converging to $H^*(F_\phi; \K)$ does not collapse at the $E_2$-term if
 $H(\phi)_\K$ is non-zero. Therefore we see that the pair $(\phi, *)$ with
 the constant map $* \to S^d$ is not relatively $\K$-formalizable if $H(\phi)_\K\neq 0$, 
 even though $S^d$ and $S^{2d-1}$ are $\K$-formal. Observe that the map $\phi$ 
 satisfies neither of the conditions (i) and (ii) in Proposition 
 \ref{prop:formalizable}. 
\end{rem}

\medskip
\noindent
{\it  Proof of Theorem \ref{thm:realization}.}
Recall from Theorem \ref{thm:J} 
the cohomology of the molecule $\Sigma^{-l}Z_m$ ($m \geq 0$). 
Suppose that $d+l =0$. 
It is immediate that $-m(d-1)+l <0$. Thus if $\Sigma^{-l}Z_m$ is
 realizable, then $-m(d-1)+l = 0$ so that $H^*(\Sigma^{-l}Z_m)=
H^*(\Sigma^{-m(d-1)}Z_m)\cong H^*(S^{(m+1)d-m}; \K)$ as a vector space. 


Suppose that $\Sigma^{-m(d-1)}Z_m$ is realized 
 by a finite CW complex $X$ with a map $\phi : X \to S^d$. 
We then claim that $m=0$ or $m=1$ and $d$ is even. 
The $i$th integral cohomology of $X$ is finitely generated
 for any $i$.  We see that  
$H^*(X)\otimes \K= H^*(X;
 \K)=H^*(\Sigma^{-m(d-1)}Z_m)=H^*(S^{(m+1)d-m};\K)\cong 
\K\oplus \Sigma^{-(m+1)d+m}\K$ and hence the rank of the $((m+1)d-m)$th
 integral homology of $X$ is at most one. It follows that
$H^*(X; \Q)=\Q\oplus \Sigma^{-(m+1)d+m}\Q$ or 
$H^*(X; \Q)=\Q$. 

Let  $\{\overline{E}_r, \overline{d}_r\}$ be 
the EMSS converging to $H^*(F_\phi; \Q)$. In view of the Koszul
resolution of $\K$ as an $H^*(S^d; \K)$-module, we see that  
$$
\overline{E}_2^{*,*}\cong 
\left\{
\begin{array}{l}
\wedge (s^{-1}x_d)\otimes \Q[\tau]\otimes H^*(X;\Q)  
\ \ \text{if} \ d \ \text{is even},  \\
\Q[s^{-1}x_d]\otimes H^*(X;\Q)   \ \  \text{if} \ d \ \text{is odd}, 
\end{array}
\right.
$$
where $\text{bideg} \ \tau = (-2, 2d)$ and 
$\text{bideg} \ s^{-1}x_d=(-1, d)$. Therefore, if $d$ is odd, then 
 the dimension of $H^*(F_\phi;\Q)$ is infinite because $s^{-1}x_d$ is a
 permanent cycle for dimensional reasons.  
Suppose that $d$ is even and $m > 1$.  
Since $(m+1)d-m \geq 3d-2 > 2d-1$, it follows that 
the element $\tau$ is a permanent cycle and hence 
$\dim H^*(F_\phi;\Q)=\infty$. 

The cohomologies $H^i(X; {\mathbb Z})$ and 
$H^i(\Omega S^d; {\mathbb Z})$ are finitely generated for any $i$. 
By considering the Leray-Serre spectral sequence of the fibration 
$\Omega S^d \to F_\phi \to X$, 
we see that $H^i(F_\phi; {\mathbb Z})$ is also finitely generated 
for any $i$. 
This implies that  $\dim H^*(F_\phi;\K)=\infty$. 
Thus we conclude from Proposition \ref{prop:compact} that if  
$\Sigma^{-m(d-1)}Z_m$ is realizable, then $m=1$ and $d$ is even or $m=0$. 

In order to complete the proof, it suffices to show that 
$\Sigma^{-(d-1)}Z_1$ is realizable if $d$ is even. 
In that case, for the Whitehead product 
$[\iota, \iota] : S^{2d-1} \to S^d$ 
of the identity map $\iota : S^d \to S^d$, it is well-known that 
$H([\iota, \iota])=\pm 2$; see \cite[Chapter 4]{MT}. 
Proposition \ref{prop:spheres} implies that for 
the irreducible map  $\alpha : Z_0 \to \Sigma^{-(d-1)}Z_1$,   
there exists an isomorphism 
$\Psi : \Sigma^{-(d-1)}Z_1 \to C^*(S^{2d-1}; \K)$ 
which fit into the commutative diagram
$$ 
\xymatrix@C15pt@R5pt{
         & \Sigma^{-(d-1)}Z_1 \ar[dd]^{\Psi}_{\cong} \\
Z_0=C^*(S^d; \K) \ar[ru]^{\alpha} \ar[rd]_{[\iota, \iota]^*}  & \\
         & C^*(S^{2d-1}; \K)
}
$$
in $\D(C^*(S^{d}; \K))$ up to scalar multiple. Thus we have  
$\Psi \alpha = k [\iota, \iota]^*$ for some non-zero element $k \in \K$.   
It turns out that the molecule $\Sigma^{-(d-1)}Z_1$ is realizable. 
This completes the proof. 
\hfill\qed

\begin{rem}
There exists an element of Hopf invariant one 
in $\pi_{2d-1}(S^d)$ if $d=2, 4$ or $8$. Therefore, the proof of 
Theorem \ref{thm:realization} allows us to conclude that
 the indecomposable element $\Sigma^{-(d-1)}Z_1$ is realizable with 
$S^{2d-1}$ in $\D^c(C^*(S^d; \K))$ for any field $\K$ if $d=2, 4$ or $8$.  
\end{rem}

\section{Computational examples}

Recall the functor 
$F_{S^d} : \text{D}(C^*(S^d; \K)) \to \text{D}(H^*(S^d; \K))$
described in Section 4, which gives an equivalence between 
triangulated categories. 
In order to prove Proposition \ref{prop:bundle}, 
we need a lemma concerning this functor. 

\begin{lem}
\label{lem:key2} Suppose that $d$ is even. Then, 
in $\text{\em D}^c(H^*(S^d; \K))$, 
$$
F_{S^d}(\Sigma^{-(d-1)}Z_1) \cong (\wedge(\tau)\otimes H^*(S^d; \K), d\tau
 = x_d). 
$$
\end{lem}

\begin{proof}
The functor $F_{S^d}$ leaves the cohomology of an object
 unchanged. Remark \ref{rem:classification} implies the result.
\end{proof}

\medskip
\noindent
{\it Proof of Proposition \ref{prop:bundle}.} 
By assumption, the cohomology $H^*(BG; \K)$ is a polynomial algebra
generated by elements with even degree, say  
$$
H^*(BG;\K)\cong \K[x_1, x_2, ..., x_l], 
$$
where $\deg x_1 \leq \deg x_2 \leq \cdots \leq \deg x_l$ and 
each $\deg x_i$ is even. Since $G$ is simply-connected, it follows that 
$\deg x_1\geq 4$.  Moreover,  $\widetilde{H}^i(S^4;\K)$ is
nonzero if and only if $i = 4$, and  
$\dim \widetilde{H}^{4-1}(\Omega BG; \K)-\dim (QH^*(BG; \K))^4=0$. 
Therefore Proposition \ref{prop:formalizable} allows us to deduce 
that the pair $(f, \pi)$ of maps is relatively
$\K$-formalizable, where $\pi : EG \to BG$ denotes the projection of the
universal $G$-bundle. 
Theorem \ref{thm:main} implies that  
$$\text{level}_{S^4}(E_f)
=\text{level}_{\D(H^*(S^4; \K))}
(\K\otimes^\text{L}_{H^*(BG; \K)}H^*(S^4; \K))=:L.
$$ 

Consider the case where $H^4(f; \K)\neq 0$. 
Without loss of generality, we assume that $H^4(f; \K)(x_1)=z_4$ and 
   $H^*(f; \K)(x_j)=0$ for $j \neq 1$.  
Here $z_4$ is the generator of the
algebra $H^*(S^4; \K)$ of degree $4$. 
We then have 
$$M:=\K\otimes^\text{L}_{H^*(BG; \K)}H^*(S^4; \K)\cong \wedge 
((s^{-1}x_2, ...., s^{-1}x_l), 0)\otimes (\wedge (s^{-1}x_1)\otimes
H^*(S^4; \K), \delta), 
$$
in $\D^c(H^*(S^4;\K))$, where $\delta s^{-1}x_1= z_4$. 
It follows from Lemma \ref{lem:key2} that 
$M \cong \wedge ((s^{-1}x_2, ...., s^{-1}x_l), 0)\otimes
F_{S^4}(\Sigma^{-(4-1)}Z_1)$.
This fact yields  
that $M$ is isomorphic to a coproduct of the molecule 
$F_{S^4}(\Sigma^{-(4-1)}Z_1)$ and certain shifts as an 
$H^*(S^4; \K)$-module. 

The functor $F_{S^4}$ is exact and 
gives an equivalence between the triangulated
categories $\text{D}(C^*(S^4; \K))$ and $\text{D}(H^*(S^4; \K))$. 
By \cite[Theorem 2.4(6)]{ABIM} and Proposition \ref{prop:Z}, we see that 
$L=\text{level}_{\text{D}(C^*(S^4; \K))}\Sigma^{-(4-1)}Z_1=2$.

Suppose that $\widetilde{H}^*(f;\K)= 0$. It follows that 
$\K\otimes^\text{L}_{H^*(BG; \K)}H^*(S^4; \K)$ is isomorphic to the
DG module $\wedge (s^{-1}x_1, s^{-1}x_2, ...., s^{-1}x_l)
\otimes H^*(S^4; \K)$ with the trivial differential, which is a
coproduct of $H^*(S^4; \K)$ and certain shifts. We conclude that $L=1$. 
\hfill\qed

\medskip
\noindent
{\it Proof of Proposition \ref{prop:variation}.} 
We observe that $(g, \pi)$ is a relatively 
$\K$-formalizable pair. Indeed the maps $g$ and $\pi$ satisfy the
conditions (ii)
and (i), respectively. Thus 
Theorem \ref{thm:main} yields that the $C^*(S^4;\K)$-level of $C^*(E_g; \K)$
is equal to the $H^*(S^4; \K)$-level of 
$H^*(BH; \K)\otimes^{\mathbb L}_{H^*(BG;\K)}H^*(S^4; \K)$.  
Since $H$ is a maximal rank subgroup of $G$, 
it follows from \cite[6.3 Theorem]{Baum} that 
$H^*(BH;\K)$ is a free $H^*(BG;\K)$-module. 
Therefore 
$H^*(BH; \K)\otimes^\text{L}_{H^*(BG;\K)}H^*(S^4; \K)$ is
isomorphic to a coproduct of shifts of $H^*(S^4; \K)$. This completes
the proof. 
\hfill\qed

\begin{ex}
\label{ex:non-formal}
Let $E_\nu \to S^7$ be the fibration described 
in Remark \ref{rem:EMSS-K-formal}, namely 
the pullback of the Hopf map $\nu : S^7 \to S^4$ by itself. 
We here compute the level of $E_\nu$.  

Consider the commutative diagram
$$
\xymatrix@C15pt@R8pt{
S^3 \ar[rd]\ar[dd] \ar[rr] &        & S^7  \ar@{->}'[d][dd] \ar@{=}[rd] & \\
             &  E_\nu \ar[dd]_(0.3)\pi \ar[rr] &                 & S^7
 \ar[dd]^\nu \\ 
 pt \ar[rd]\ar@{->}'[r][rr] & & S^4 \ar@{=}[rd] &  \\  
            & S^7 \ar[rr]_\nu & & S^4. 
}
$$
Let $\{E_r, d_r\}$ and $\{\overline{E}_r, \overline{d}_r \}$ be the
 Eilenberg-Moore spectral sequences  
for the front square and the back square, respectively. 
Then the diagram above gives rise to a morphism 
$\{g_r\} : \{E_r, d_r\} \to \{\overline{E}_r, \overline{d}_r \}$ 
of the spectral sequences. 
Observe that 
$\overline{E}_2\cong H^*(S^7; \K)\otimes \Gamma[w]\otimes
 \wedge(s^{-1}x_4)$
and $E_2\cong H^*(S^7; \K)\otimes \Gamma[w]\otimes
 \wedge(s^{-1}x_4)\otimes H^*(S^7; \K)$, 
where $\text{bideg} \ w=(-2, 8)$ and 
$\text{bideg} \ s^{-1}x_4=(-1, 4)$. Moreover it follows that 
$g_2(w)=w$, $g_2(s^{-1}x_4)=s^{-1}x_4$, $g_2(x)=x$ for 
$x \in H^*(S^7; \K)\otimes 1 \otimes 1\otimes 1$ and 
$g_2(y)=0$ for $y \in 1\otimes 1\otimes 1\otimes H^*(S^7; \K)$. 

By the same argument as in the proof of Proposition \ref{prop:spheres}, 
we see that $\overline{d}_2(\gamma_i(w))=x_7\gamma_{i-1}(w)$. 
This implies that $d_2(\gamma_i(w))=x_7\gamma_{i-1}(w)$ and hence 
$E_\infty\cong E_3^{*,*}\cong \wedge(s^{-1}x_4)\otimes H^*(S^7; \K)$ as an 
$H^*(S^7; \K)$-module. For dimensional reasons, there is no 
extension problem. Thus it follows that  
$H^*(E_\nu)\cong \wedge(s^{-1}x_4)\otimes H^*(S^7; \K)$ as an 
$H^*(S^7; \K)$-module. 
We observe that, by Remark \ref{rem:EMSS-K-formal},  
the pair $(\nu, \nu)$ of maps is {\it not} relatively
 $\K$-formalizable.

Define a $C^*(S^7; \K)$-module map 
$\varphi : \Sigma^{-3}C^*(S^7; \K) \to C^*(E_\nu; \K)$ by 
$\varphi(\Sigma^{-3}z)=s^{-1}x_4' \pi^*(z)$, where $s^{-1}x_4'$ is a
 representative element of $s^{-1}x_4 \in H^*(E_\nu; \K)$. 
We see that the map 
$
\varphi\oplus \pi^* : \Sigma^{-3}C^*(S^7; \K)\oplus C^*(S^7; \K) \to 
C^*(E_\nu ; \K)
$ 
is a quasi-isomorphism. The fact allows us to conclude that 
$
\text{level}_{S^7}(E_\nu)=1.  
$
\end{ex}

\begin{ex}
\label{ex:ex} 
We denote by $\Sigma^i{\mathbb Z}A_\infty$ the connected component of the
 Auslander-Reiten quiver containing $\Sigma^iZ_0$ in
 $\D^c(C^*(S^d;\K))$, where $0\leq i \leq d-2$. 

Let $G_2$ be the compact simply-connected simple Lie group of type $G_2$. 
 Consider the principal $G_2$-bundle $G_2 \to X_1 \to S^4$ with the classifying
 map $f : S^4 \to BG_2$ which represents a generator of 
$\pi_4(BG_2)\cong \pi_3(G_2)\cong {\mathbb Z}$. 
It is well-known that 
$H^*(BG_2; {\mathbb F}_2)\cong {\mathbb F}_2[y_4, y_6, y_7]$, where
 $\deg y_i = i$. Therefore, 
it follows from a computation similar to that in the proof of Proposition 
\ref{prop:bundle} that, in $\D^c(C^*(S^4; {\mathbb F}_2))$,  
$$
C^*(X_1; {\mathbb F}_2)\cong \Sigma^{-3}Z_1\otimes  {\mathbb
 F}_2\{s^{-1}y_6, s^{-1}y_7\}\cong \Sigma^{-3}Z_1 \oplus \Sigma^{-3-5}Z_1
 \oplus \Sigma^{-3-6}Z_1.  
$$
This yields that $C^*(X_1; {\mathbb F}_2)$ consists of two molecules 
$\Sigma^{-3}Z_1$ and $\Sigma^{-3-6}Z_1$ in ${\mathbb Z}A_\infty$
 and one molecule $\Sigma^{-3-5}Z_1$ in 
$\Sigma^{2}{\mathbb Z}A_\infty$. We see that 
$\text{level}_{S^4}(X_1)=2$. 

Consider the principal $SU(4)$-bundle $SU(4)\to X_2 \to S^4$ with the
 classifying map representing the generator of 
$\pi_4(BSU(4))\cong {\mathbb Z}$. We observe that 
$H^*(BSU(4); {\mathbb F}_2)\cong {\mathbb F}_2[c_2, c_3, c_4]$, 
where $\deg c_i=2i$.  A computation similar to that above 
 enables  us to conclude that 
$$
C^*(X_2; {\mathbb F}_2)\cong \Sigma^{-3}Z_1 \oplus \Sigma^{-3-5}Z_1
 \oplus \Sigma^{-3-7}Z_1.  
$$
Observe that the molecules $\Sigma^{-3}Z_1$, $\Sigma^{-3-5}Z_1$ and 
$\Sigma^{-3-7}Z_1$ are in the quivers ${\mathbb Z}A_\infty$, 
$\Sigma^{2}{\mathbb Z}A_\infty$ and  $\Sigma^{1}{\mathbb Z}A_\infty$, 
respectively. This yields that   
$\text{level}_{S^4}(X_2)=2$.
\end{ex}

\medskip
\noindent
{\it Acknowledgments.} I thank Ryo Takahashi 
for helpful discussions on Auslander-Reiten theory and the levels of
modules. I also thank Dai Tamaki for a comment on the realization
problem of molecules without which Proposition \ref{prop:spheres} could
not have been obtained. I am grateful to Peter J{\o}rgensen 
and Jean-Claude Thomas for valuable suggestions and comments 
to revise the first version of this paper. 
I am also deeply grateful to the referee for careful reading of  a previous version of this paper 
and for many helpful comments to revise the version.  

\section{Appendix}

We recall briefly the $TV$-model introduced by Halperin and 
Lemaire \cite{H-L}.  

Let $TV$ be the tensor algebra 
$\sum_{n \geq 0}V^{\otimes n}$ on a graded vector space $V$ over a field 
${\mathbb K}$ and  let $T^{\geq k}V$ denote its ideal 
$\sum_{n \geq k}V^{\otimes n}$ of the algebra $TV$, where $V^{\otimes
0}=\K$.  
As usual, we define the degree  of the element $w = v_1v_2 \cdots v_l \in TV$ 
by $\deg w = n_1 + \cdots +n_l$ if $v_{n_i} \in V^{n_i}$.

Let $V'$ and $V''$ be copies of $V$. We write $sv$ for the element of 
$\Sigma V$ corresponding to $v \in V$.     
The cylinder object $TV\wedge I=(T(V'\oplus V''\oplus \Sigma V), d)$ 
introduced by Baues and Lemaire \cite[\S1]{B-L} is a DG algebra with 
differential $d$ defined by 
$$
dv'=(dv)', \  dv'' = (dv)'' \ \text{and} \ dsv = v'' - v' -S(dv), 
$$
where 
$S: TV \to T(V'\oplus V''\oplus \Sigma V)$ is a map with 
$Sv = sv$ for $v \in V$ and $S(xy)= Sx\cdot y'' + (-1)^{\deg x}x'\cdot Sy$ for 
$x, y \in TV$.  
The inclusions $\e_0 : TV \to TV\wedge I$ and 
$\e_1 : TV \to TV\wedge I$ are defined by $\e_0(v)=v'$ and
$\e_1(v)=v''$, respectively. 

For DG algebra maps $\phi', \phi'' : TV \to A$,  
we say that $\phi'$ and $\phi''$ are homotopic if the DG algebra map 
$(\phi', \phi'') : T(V'\oplus V'') \to A$ extends to a DG algebra map 
$\Phi : TV\wedge I \to A$; that is $\phi'=\Phi\e_0$ and
$\phi''=\Phi\e_1$. We refer the reader to \cite[Section 3]{FHT2} 
for the homotopy theory of DG algebras.

A {\it $TV$-model} for a differential graded algebra $(A, d_A)$ 
is a quasi-isomorphism  $(TV ,d) \to (A, d_A)$. 
Moreover the model is  called  
minimal if $d(V) \subset T^{\geq 2}V$. 
For any simply-connected space whose cohomology with coefficients 
in $\K$ is locally finite, there exists a minimal 
$TV$-model $(TV ,d) \to C^*(X; {\mathbb K})$ which is unique up to homotopy.  
Such a model $(TV, d)$ is called a  {\it  minimal model for} $X$. 
It is known that the vector space $V^n$ is isomorphic to 
$(\Sigma^{-1}\tilde{H}^*(\Omega X;\K))^n=\tilde{H}^{n-1}(\Omega X; {\mathbb K})$ and 
the quadratic part of the differential $d$ is the coproduct on 
$\tilde{H}^*(\Omega X; \K)$ up to the isomorphism 
$V \cong \Sigma^{-1}\tilde{H}^*(\Omega X;\K)$. 
The reader is referred to \cite{H-L} and \cite[Introduction]{B-T} 
for these facts and more details of $TV$-models.


\end{document}